\documentclass[final,letterpaper,oneside]{amsart}
\usepackage[l2tabu,orthodox]{nag}
\usepackage[all]{onlyamsmath}
\usepackage[utf8]{inputenc}
\usepackage[T1]{fontenc}
\usepackage{lmodern}
\usepackage{microtype}
\usepackage{geometry}
\usepackage[numbers,sort]{natbib}
\usepackage{amsmath,amssymb,amsthm}
\usepackage{mathtools,mathrsfs,centernot,multicol}
\usepackage[dvipsnames]{xcolor}

\usepackage[shortlabels]{enumitem}
\setlist[enumerate,1]{label=\textup{(\arabic*)}} 
\newlist{enumroman}{enumerate}{1}
\setlist[enumroman]{label={\textup{(\roman*)}}, ref={\textup{\thetheorem(\roman*)}}}

\usepackage[final,pdfusetitle]{hyperref}
\hypersetup{colorlinks=true, linkcolor=blue, citecolor=blue, urlcolor=blue}  
\usepackage{uri}
\usepackage{bookmark}
\bookmarksetup{depth=subsection,open,openlevel=2,numbered}

\usepackage[capitalise,noabbrev]{cleveref}
\crefdefaultlabelformat{#2\textup{#1}#3}
\theoremstyle{plain}
\newtheorem{theorem}{Theorem}[section]
\newtheorem{lemma}[theorem]{Lemma}

\newtheorem{proposition}[theorem]{Proposition}

\theoremstyle{definition}
\newtheorem{definition}{Definition}[section]
\numberwithin{equation}{section}

\newcommand{\R}{\mathbb R}
\newcommand{\Q}{\mathbb Q}
\newcommand{\Z}{\mathbb Z}
\newcommand{\N}{\mathbb N}
\newcommand{\X}{\mathbb X}
\newcommand{\abs}[1]{\lvert #1\rvert}
\newcommand{\angles}[1]{\langle #1\rangle}
\renewcommand{\phi}{\varphi}
\renewcommand{\bar}{\overline}
\newcommand{\dsum}{\bigoplus}
\newcommand{\sm}{\setminus}
\newcommand{\sub}{\subseteq}
\newcommand{\psub}{\subset}
\newcommand{\usim}{\mathord{\sim}}
\newcommand{\ull}{\mathord{\ll}}
\renewcommand{\nsim}{\not\sim} 
\newcommand{\nll}{\centernot\ll} 
\DeclareMathOperator{\aut}{Aut}
\DeclareMathOperator{\CB}{CB}

\begin{document}
\title{Space of orders with finite Cantor-Bendixson rank}
\author{Waseet Kazmi}
\address{}
\email{waseetkazmi@outlook.com}

\subjclass[2020]{Primary 06F15, 20F60}
\keywords{Ordered groups, space of orders, Cantor-Bendixson rank, semidirect product} 

\begin{abstract}
The goal of this paper is to show the following result: For every integer $n\geq 2$ there is a countable orderable group such that its space of orders is countable and has Cantor-Bendixson rank $n$. We show this by explicitly constructing a family of orderable groups with the desired properties.
\end{abstract}

\maketitle

\section{Introduction}
Let $G$ be a group and let $<$ be a strict total order on $G$. The pair $(G,<)$ is called an \textit{ordered group} if the order $<$ is both left and right invariant, that is, $a<b$ implies $ca<cb$ and $ac<bc$ for all $a,b,c\in G$. 
A group is called an \textit{orderable group} if it admits a strict total order which is bi-invariant.
We can equivalently define a group $G$ to be orderable if it admits a subset $P\sub G$, called the \textit{positive cone}, satisfying:
\begin{enumerate}
\item $P \sqcup P^{-1} \sqcup \{1\}= G$;
\item $P\cdot P \sub P$; and
\item $gPg^{-1}\sub P$ for all $g\in G$.
\end{enumerate}
There is a correspondence between the orderings and the positive cones of $G$. Given an ordering $<$, we can define the set $P=\{g\in G\mid 1<g\}$ to be the positive cone. Conversely, given a positive cone $P$, we can define an invariant strict total order on $G$ by the recipe $g<h$ if and only if $g^{-1}h\in P$ for all $g,h\in G$.
We define the \textit{space of orders} of an orderable group $G$, denoted $\X(G)$, to be the set of all positive cones of $G$. In notation,
\[ \X(G)=\{P\sub G \mid P\text{ is the positive cone of an order on } G\}.\]
The space of orders $\X(G)$ is a subspace of $2^G$, where $2^G$ is endowed with the usual product topology. More so, it can be shown that $\X(G)$ is a closed subset of $2^G$ and since $2^G$ is compact, it readily follows that $\X(G)$ is also compact. By a slight abuse of notation, we will use the correspondence between positive cones and orderings to write ${<}\in \X(G)$ whenever no confusion will arise from doing so.

Here is another useful way to describe the topology of $\X(G)$.
Suppose ${<}$ is an ordering of $G$. Consider a finite string of inequalities $g_1< \cdots<g_n$ for $g_i\in G$. The set of orderings of $G$ that satisfy these inequalities forms an open neighborhood of $<$ in $\X(G)$. 
More so, the set of all such neighborhoods forms a basis for the topology of $\X(G)$. Equivalently, one can multiply the inequalities as necessary and see that a basis for $\X(G)$ consists of all sets of orderings in which some specified finite set of elements of $G$ are all positive.
 
In an ordered group $(G,<)$, the \textit{absolute value} $\abs g$ of an element $g\in G$ is defined to be $\abs g = \max(g, g^{-1})$.
We now introduce two important relations that will be central to our discussion in proving our main result.
\begin{definition} 
Let $(G,<)$ be an ordered group and let $g,h \in G$. We define $g\sim h$ if there exist positive integers $m$ and $n$ such that $\abs g < \abs h^m$ and $\abs h <\abs g^n$.
We say $g$ and $h$ are \textit{Archimedean equivalent} if $g\sim h$.
We say $g$ is \textit{Archimedean less than} $h$, denoted by $g\ll h$, if $\abs g^n<\abs h$ for all positive integers $n$. 
\end{definition}
We call an ordered group \textit{Archimedean} if all nonidentity elements are Archimedean equivalent. 
A well-known fact is that every Archimedean ordered group is abelian.
We will refer to the equivalence classes of $G$ under the Archimedean equivalence relation as the \textit{Archimedean classes} of $G$.  We will write $[g]$ to denote the Archimedean class of $g\in G$.
We can define a linear order on the set of Archimedean classes of $G$ by declaring $[g]\ll [h]$  if and only if  $g\ll h$ for all $g,h\in G$. 
(We will abuse notation and denote the order on the Archimedean classes also by $\ull$.)
It follows from the definitions that the identity of $G$ forms an Archimedean class by itself---moreover, it is the least class under the prescribed ordering of the Archimedean classes.

Recall that if $X$ is a topological space, a point $x\in X$ is called an \textit{isolated point} if $\{x\}$ is open and is called a \textit{limit point} if every open neighborhood of $x$ contains a point other than $x$.
If $A\sub X$, we say that a point $x\in X$ is a \textit{limit point of $A$} if every open neighborhood of $x$ contains a point of $A$ other than $x$.
In the context of the space of orders, an isolated point in $\X(G)$ corresponds to an order on $G$ that is the unique order satisfying some finite string of inequalities. Equivalently, an isolated order is the unique order in which some fixed finite set of elements of $G$ are all positive.

The \textit{Cantor-Bendixson derivative} of a topological space $X$, denoted $X'$, is the set of nonisolated points of $X$; equivalently, $X'$ is the set of limit points of $X$. For each ordinal $\alpha$, define $X^{(\alpha)}$ recursively as follows:
\begin{enumerate}
\item $X^{(0)}=X$,
\item $X^{(\alpha+1)}= (X^{(\alpha)})'$, 
\item $X^{(\alpha)}= \bigcap_{\gamma<\alpha} X^{(\gamma)}$ if $\alpha$ is a limit ordinal.
\end{enumerate}
In general, if $X$ is an arbitrary topological space, then $X'$ is not necessarily closed. 
However, if $X$ is a Hausdorff space, then $X'$ is closed in $X$ and therefore $X^{(\alpha)}$ is a closed subset of $X$ for all $\alpha$.
It is clear from the definition that the subspaces $X^{(\alpha)}$ form a nonincreasing sequence, that is, $X^{(\delta)}\sub X^{(\gamma)}$ whenever $\gamma<\delta$.
It can be shown that this transfinite sequence of derivatives must eventually be constant.
\begin{proposition} \label{prop:CBrankK}
Let $X$ be a topological space. If $\abs X=\kappa$ for some cardinal number $\kappa$, then there is an ordinal $\alpha<\kappa^+$ such that $X^{(\alpha)}=X^{(\alpha+1)}$. 
\end{proposition}
\begin{proof}
Assume for a contradiction, $X^{(\alpha)}\neq X^{(\alpha+1)}$ for all ordinals $\alpha<\kappa^+$.  Then for all $\alpha<\kappa^+$, there exists some $x_{\alpha}\in X$ such that $x_\alpha\in X^{(\alpha)}\sm X^{(\alpha+1)}$. 
This gives a map from $\kappa^+$ into $X$ via $\alpha\mapsto x_\alpha$ that is injective. This is a contradiction since by assumption $\abs X=\kappa$.  Hence, there must exist some $\alpha<\kappa^+$ such that $X^{(\alpha)}= X^{(\alpha+1)}$.
\end{proof}

\begin{definition}
Let $X$ be a topological space. 
The \textit{Cantor-Bendixson rank} of $X$, denoted $\CB(X)$, is the least ordinal $\alpha$ such that $X^{(\alpha)}=X^{(\alpha+1)}$. 
We say a point $x\in X$ has \textit{Cantor-Bendixson rank} $\alpha$ if $\alpha$ is the least ordinal such that $x\in X^{(\alpha)}$ but $x\notin X^{(\alpha+1)}$. 
We write $\CB(x)=\alpha$ to denote $x$ has Cantor-Bendixson rank $\alpha$.
\end{definition}
\begin{proposition}
Let $X$ be a topological space and let $A\sub X$. Suppose $\CB(a)\geq \alpha$ for all $a\in A$, or equivalently, $A\sub X^{(\alpha)}$. If $x\in X$ is a limit point of $A$, then $\CB(x)\geq \alpha+1$.
\end{proposition}
\begin{proof}
If $A\sub X^{(\alpha)}$, then $A'\sub ( X^{(\alpha)} )'=X^{(\alpha+1)}$. Since $x$ is a limit point of $A$, we have that $x\in A'\sub X^{(\alpha+1)}$ and therefore $\CB(x)\geq \alpha+1$.
\end{proof}
We will be interested in the case when $G$ is a countable group and $\X(G)$ is a countable subset of $2^G$. The following proposition will be useful in this case.
\begin{proposition} 
If $C$ is a countable closed subset of $2^\N$, then the Cantor-Bendixson rank of $C$ is a countable ordinal $\alpha$ such that $\alpha$ is a successor ordinal and $C^{(\alpha)}=\emptyset$.
\end{proposition}
\begin{proof}
Let $C$ be a nonempty countable closed subset of $2^{\N}$. By \cref{prop:CBrankK}, the Cantor-Bendixson rank of $C$ is a countable ordinal, say $\alpha$.  We claim that $C^{(\alpha)}=\emptyset$. Since $C$ is countable and closed, it must contain an isolated point. Because otherwise $C$ would be a countable, perfect subset of $2^{\N}$ and that is impossible since perfect sets are uncountable.  
In addition, every nonempty derivative of $C$ must also contain an isolated point because they are also closed, and either finite or countable.  Thus $C^{(\alpha)}= C^{(\alpha+1)}$ implies that $C^{(\alpha)}=\emptyset$, as required. 

Next we want to show that $\alpha$ is a successor ordinal. Suppose for a contradiction $\alpha$ is a limit ordinal. Then $C^{(\alpha)}= \bigcap_{\gamma<\alpha}C^{(\gamma)}$ where each $C^{(\gamma)}$ is a nonempty closed set. It follows by compactness that $\bigcap_{\gamma<\alpha}C^{(\gamma)}$ must be nonempty and we have a contradiction. Therefore $\alpha$ is a successor ordinal.
\end{proof} Observe that if $C\sub 2^\N$ is a nonempty countable closed set and $C$ has Cantor-Bendixson rank $\alpha$, where $\alpha=\beta+1$ for some ordinal $\beta$, then $C^{(\beta)}$ will be nonempty, finite and contain only isolated points. 

The notion of Cantor-Bendixson rank has been used in the literature to prove various results about the space of orders. For example, in \cite{linnell2011}, it is shown that the number of left orders of a group is either finite or uncountable. 
The Cantor-Bendixson rank is also related to the study of dense and discrete orderings of a group, see for example \cite{clay2010,clay2020}.
The goal of this paper is to prove that for every integer $n\geq 2$ there is a countable, orderable group $G_n$ such that its space of orders $\X(G_n)$ is countable and has Cantor-Bendixson rank $n$.
The first example in the literature of an orderable group with countably many orders is given by Buttsworth in \cite{buttsworth1971}. 
Buttsworth constructs a countable orderable group using semidirect products and exploits the conjugation actions inside the semidirect products to show that the group in question has exactly countably many orders. Buttsworth does not discuss the Cantor-Bendixson rank of the space of orders of the group that is constructed.
We will generalize the construction given in \cite{buttsworth1971} to build a family of orderable groups with the desired properties as stated in our goal above.

\section{Preliminaries}
We state some preliminary results that we will need. 
The first result gives a sufficient condition for orderability of a semidirect product.
\begin{proposition} \label{prop:semiorderable}
Let $(N,<_N)$ and $(H,<_H)$ be ordered groups. Let $G=N\rtimes_\phi H$ be a semidirect product with $\phi\colon H\to \aut (N)$. Suppose $\phi(h)\colon N\to N$ is order-preserving with respect to $<_N$ for all $h\in H$. 
Then $G$ is orderable and the following recipe defines an ordering 
\[ 1<_G nh \text{ if and only if } 1<_H h \text{, or else } h=1 \text{ and } 1<_N n.\]
\end{proposition}
\begin{proof}
The proof of this follows from a routine verification that our prescribed ordering is bi-invariant under the given assumptions.
\end{proof}
We collect several useful facts about the relations $\usim$ and $\ull$. The proofs of these can be found in \cite[\S1.5]{kazmi2023}.
\begin{proposition}
Let $(G,<)$ be an ordered group. 
\begin{enumroman}
\item \label[prop]{prop:archlessthanpos}
If $g\nsim h$, then $\abs g<\abs h$ if and only if $g\ll h$.
\item  If $g\ll h$, then $h,gh$ and $hg$ all have the same sign, that is to say, all elements are either positive or negative.
\item If $a\ll g$ and $b\ll g$, then $ab\ll g$.
\item If $a\ll b$, then $b\sim ab\sim ba$.
\item If $g\nsim h$, $a\ll g$ and $b\ll h$, then $ab\ll gh$.
\item \label[prop]{prop:nprodclass}
Let $a_1,\ldots, a_n\in G$. Suppose there exists an $i$ such that $a_j\ll a_i$ for all $j\neq i$.  Then $a_1\cdots a_n \sim a_i$.
\end{enumroman}
\end{proposition}
The following results are mentioned in \cite{buttsworth1971} without proof, see \cite[Proposition~3.1.1]{kazmi2023} for proofs.
\begin{proposition} 
Let $(G,<)$ be an ordered group. Suppose $x,y_i\in G$ where $i\in \Z$.
\begin{enumroman}
\item \label[prop]{prop:prelimresultsA}
If $x^{-1}y_i^{k_i}x = y_i^{l_i}$ for all $i\in \Z$ with $k_i,l_i\in \Z$ and $k_i\neq l_i$, then $y_i\ll x$ and $y_i\sim y_j$ implies $k_il_j=k_j l_i$.
\item \label[prop]{prop:prelimresultsB}
If $x^{-1}y_i x=y_{i+1}$ and $y_i\sim y_{i+1}$ does not hold, then either
\begin{gather*}
\cdots\ll y_{i-1}\ll y_i \ll y_{i+1}\ll \cdots\ll x
\shortintertext{or}
\cdots\ll y_{i+1}\ll y_i \ll y_{i-1}\ll \cdots\ll x.
\end{gather*}
\end{enumroman}
\end{proposition}

\section{Construction of \texorpdfstring{$G_n$}{Gn}}
In this section, we will be generalizing the work of Buttsworth and following the notation found in \cite{buttsworth1971}. We begin by constructing our family of orderable groups.
Let $A$ be the additive group of dyadic rationals, that is, \[A=\Big\{ \frac{m}{2^k}\mid m,k\in\Z \Big\}.\]
Let $X$ be a subset of $A$ defined by \[ X=\{x\in A\mid 0\leq x<1\}.\]
Fix an integer $n\geq 2$.  For each integer $1\leq i\leq n$,  we will define the group $H^i$.
For all $z\in\Z$ and $x\in X$, let $H^i_{z,x}$ be copies of the group $(\Q,+)$. For each $z\in \Z$, define the direct sum
\begin{gather*}
H_z^i=\dsum_{x\in X} H_{z,x}^i.
\shortintertext{Let}
H^i = \dsum_{z\in \Z}H_z^i = \dsum_{z\in \Z}\dsum_{x\in X}H_{z,x}^{i}
\shortintertext{and}
P_n=H^1\times  \cdots\times H^n. 
\end{gather*}

First, we want to construct a semidirect product $M_n = P_n \rtimes A$.
To do this, let us first fix some notation.  
We will write $h_{z,x}^{i,r} \in H_{z,x}^i$ to denote the number $r\in\Q$. 
If $r=1$, we will usually write $h_{z,x}^{i}$ instead of $h_{z,x}^{i,1}$.
Let $\lambda^\alpha$ for $\alpha \in A$ denote an arbitrary element of $A$ and let $\zeta^\beta$ for $\beta\in\Z$ denote an arbitrary element of $\Z$. 
Fix a collection of distinct prime numbers $p_1,\ldots,p_n$. 
To construct our semidirect product we define a group action of each element of $A$ on the elements $h_{z,x}^{i,r}$. The action is defined as 
\begin{equation} \label{eq:gpactionHi}
\lambda^{-\alpha}h_{z,x}^{i,r} \lambda^\alpha= h_{z,x+\alpha2^z-m}^{i,rp_i^m}
\end{equation} where $m\leq x+\alpha2^z<m+1$ for $m\in\Z$.
Of course, this action is only defined on the basic components of $P_n$ but we can extend this action to the whole group $P_n$ componentwise.  
It can be checked that this action defines a group homomorphism from $A$ into $\aut(P_n)$, thus 
we can construct the semidirect product $M_n = P_n \rtimes A$.
Next, we define an action on $M_n$ by $\Z$ via  
\begin{gather}
  \zeta^{-\beta} h_{z,x}^{i,r} \zeta^\beta = h_{z+\beta,x}^{i,r} \label{eq:gpactionZHi}
\shortintertext{and} 
  \zeta^{-\beta}\lambda^\alpha \zeta^\beta= \lambda\strut^{\frac{\alpha}{2^\beta}}. \label{eq:gpactionZA}
\end{gather}
Using the above group actions, we can finally construct the semidirect product 
\[ G_n = M_n\rtimes \Z = (P_n \rtimes A)\rtimes \Z. \]
We want to prove two facts about $G_n$: it has exactly countably many orders and its space of orders has Cantor-Bendixson rank $n$.
Our first step is showing that $G_n$ is orderable.
\begin{proposition} 
The group $G_n$ is orderable.
\end{proposition}
\begin{proof}
Consider the following order on $H_z^i$. Given an element $h\in H_z^i$, we can express it as $h_{z,x_1}^{i,r_1}\cdots h_{z,x_m}^{i,r_m}$ with $x_j\in X$ and $r_j\in\Q$. Define 
  \[ 1<_{H_z^i} h \text{ if and only if } 0<_{\R}r_1p_i^{x_1}+ \cdots+r_m p_i^{x_m} \]
where $<_\R$ denotes the usual order of $\R$ and $p_i$ is the prime fixed from earlier.
(See \cite[Appendix~A.2]{kazmi2023} for a proof of the fact that the set $\{p^x\mid x\in X\}$ is a linearly independent subset of $\R$ regarded as a $\Q$-vector space.)
The order $<_{H_z^i}$ is Archimedean and makes $H_z^i$ an Archimedean ordered group. 

We can next define an order on $H^i$. For a nonidentity element $h_{z_1}\cdots h_{z_m} \in H^i$ with $h_{z_j}\in H_{z_j}^i$ and $z_1< \cdots<z_m$, define 
\[ 1<_{H^i} h_{z_1}\cdots h_{z_m}  \text{ if and only if }  1 <_{H_{z_1}^i} h_{z_1}. \]
With an ordering of each $H^i$ fixed, we can order $P_n=H^1\times \cdots\times H^n$ lexicographically. 
Since $A$ is a subgroup of $\Q$ and the rationals can only be ordered in one of two ways as a group, we can simply fix one of these two possible orders for $A$. Likewise fix an order on $\Z$. 

It can be verified that the group action described in \eqref{eq:gpactionHi} is order-preserving with respect to the order $<_{H^i}$ on $H^i$.
So in turn, the conjugation action of $A$ upon $P_n$ preserves the given lexicographic ordering of $P_n$.
Therefore, by \cref{prop:semiorderable}, $M_n=P_n\rtimes A$ is orderable with respect to the reverse lexicographic ordering.
Lastly, it can be checked that the conjugation actions described in \eqref{eq:gpactionZHi} and \eqref{eq:gpactionZA} preserve the reverse lexicographic ordering of $M_n$.
Thus once again, by \cref{prop:semiorderable}, we can define a reverse lexicographic order on $G_n=M_n\rtimes \Z$. Hence, we can finally conclude that $G_n$ is orderable.
\end{proof}

Our next step is showing that the subgroups $H_z^i$ inherit unique orders up to duals from $G_n$. Note that two linear orders are said to be dual if they are reverse of each other.
\begin{lemma}
\label{lem:HizOrder}
In any order of $G_n$, the order of each subgroup $H_z^i$ is Archimedean and unique up to duals. 
\end{lemma}
\begin{proof}
Fix an ordering $<_G$ of $G_n$ and fix $z\in \Z$. Suppose $x_1,x_2\in X$ with $x_1<x_2$. Then $0<x_2-x_1<1$ and so we may write $x_2-x_1= m/2^k$ with $m<2^k$ and $m,k\in\Z^+$. Let $\alpha= m/2^{k+z}\in A$. Observe that
\begin{gather}
\lambda^{-2^k\alpha }h_{z,x_1}^i \lambda^{2^k\alpha}=h_{z,x_1}^{i,p_i^m} \label{eq:identityA} 
\shortintertext{and}
\lambda^{-\alpha}h_{z,x_1}^i \lambda^{\alpha} = h_{z,x_2}^i.  \label{eq:identityB} 
\end{gather}
Suppose on the contrary that $H_z^i$ is not Archimedean. Let $\ll$ denote the induced linear order on the Archimedean classes under $<_G$.
Without loss of generality, assume $h^i_{z,x_1}\ll h^i_{z,x_2}$. If we conjugate both sides by the element $\lambda^\alpha$ from \eqref{eq:identityB} we get
  \[h_{z,x_2}^i\ll \lambda^{-2\alpha}h_{z,x_1}^i \lambda^{2\alpha}.\] 
Since the relation $\ll$ is transitive, we can conclude that 
  \[h^i_{z,x_1}\ll \lambda^{-2\alpha}h^i_{z,x_1}\lambda^{2\alpha}.\]
If we continue conjugating by $\lambda^\alpha$, we will arrive at the relation
$h^i_{z,x_1}\ll \lambda^{-2^k\alpha}h^i_{z,x_1}\lambda^{2^k\alpha}$. But according to \eqref{eq:identityA}, in fact, $h^i_{z,x_1} \ll h_{z,x_1}^{i,p_i^m}$. But this is a contradiction since $\Q$ only has two orders and both of those orders are Archimedean. 
Thus it must be the case that $h^i_{z,x_1}$ and $h^i_{z,x_2}$ are Archimedean equivalent. Since $x_1$ and $x_2$ were arbitrary elements of $X$, it follows that $H^i_z$ is an Archimedean ordered group with respect to any order it inherits from $G_n$.

Next, we want to show that any order of $G_n$ restricted to $H^i_z$ is unique up to duals.
Let $<_{H^i_z}$ be the restriction of $<_G$ to $H^i_z$. 
By above $(H^i_z,<_{H^i_z})$ is Archimedean so it is order-isomorphic to a subgroup of the real numbers under their usual ordering. Let $\phi\colon H^i_z\to \R$ be the isomorphism in question, and let $L=\phi(H^i_z)$. 
Suppose $h^i_{z,0}$ is positive under $<_{H^i_z}$. We can further assume that $\phi( h^i_{z,0} )=1\in \R$. Fix $x\in X$ with $x>0$. Then if $x=\frac{m}{2^k}$, set $\alpha=\frac{m}{2^{k+z}}$.
By \eqref{eq:identityA} and \eqref{eq:identityB},
\begin{gather}
\lambda^{-2^k\alpha }h^i_{z,0}\lambda^{2^k\alpha}=h_{z,0}^{i,p_i^m}\label{eq:identityC}
\shortintertext{and}
\lambda^{-\alpha}h^i_{z,0}\lambda^{\alpha} = h^i_{z,x}.\label{eq:identityD} 
\end{gather}
Now \eqref{eq:identityD} is an order-preserving automorphism of $H^i_z$, so this gives an automorphism of $L$ which is determined by multiplication by a positive real number, say $r\in \R$. Now \eqref{eq:identityC} says $r^{2^{k}}=p_i^m$ and this implies that $r=p_i^x$. By \eqref{eq:identityD}, $\phi(h^i_{z,x})=p_i^x$. 
Therefore for all $y\in X$ and $s\in \Q$, we have $\phi( h_{z,y}^{i,s} )=s p_i^y$ 
(with $p_i^y$ always taken to be a positive real number).
That is all to say, the choice of a sign for $h^i_{z,0}$ completely determines $\phi(H^i_z)$.
So to compare elements in $H^i_z$, we can instead compare elements in $\phi(H^i_z)$. This means there is exactly one order possible on $H^i_z$ if we choose that $h^i_{z,0}$ is positive, while its dual occurs if we impose that $h^i_{z,0}$ is negative.
This shows that the order on $H^i_z$ is unique up to duals and completes the proof.
\end{proof}

To sum up the previous proof and for future convenience, we restate the unique order on $H^i_z$ that arises from the restriction of an order on $G_n$ (assuming $h^i_{z,0}$ is positive). Let $h\in H^i_z$. We can write $h=h_{z,x_1}^{i,r_1}\cdots h_{z,x_m}^{i,r_m}$ with $x_j\in X$ and $r_j\in \Q$. 
Define 
\begin{gather}\label{eq:Hizorder}
h \text{ is positive if and only if } 0<_{\R} r_1p_i^{x_1}+\cdots+r_m p_i^{x_m} 
\end{gather}
where $<_\R$ denotes the usual ordering of $\R$.  

The next two lemmas describe the relationships among the Archimedean classes of $G_n$ and how these classes are ordered.
\begin{lemma}\label{lem:Hiarchclasses}
In any order of $G_n$, for all $1\leq i \leq n$, either
\begin{gather*}
1\ll \cdots \ll  h^{i}_{z-1,0}\ll h^{i}_{z,0} \ll h^{i}_{z+1,0} \ll\cdots\ll \lambda\ll\zeta
\shortintertext{or}
1\ll \cdots \ll  h^{i}_{z+1,0}\ll h^{i}_{z,0} \ll h^{i}_{z-1,0} \ll\cdots\ll \lambda \ll\zeta
\end{gather*} holds.
\end{lemma}
\begin{proof}
Fix $1\leq i\leq n$.  
We first want to show that $h^i_{z,0}\nsim h^i_{z+1,0}$ for all $z\in\Z$. 
To simplify notation, let $p=p_i$. 
Suppose $z\geq 0$. Then
\begin{gather*} 
\lambda^{-1} h^i_{z,0}\lambda= h_{z,0}^{i,p^{2^{z}}} 
  \shortintertext{and}
\lambda^{-1} h^i_{z+1,0}\lambda= h_{z+1,0}^{i,p^{2^{z+1}}}.
\end{gather*}
Since $p^{2^z} \neq p^{2^{z+1}}$, by \cref{prop:prelimresultsA} it follows that $h^i_{z,0}\nsim h^i_{z+1,0}$.
Now, suppose $z<0$ and let $z=-k$. Then $\lambda^{-2^{k+1}} h^i_{z,0} \lambda^{2^{k+1}} = h_{z,0}^{i,p^2}$ and $\lambda^{-2^{k+1}} h^i_{z-1,0} \lambda^{2^{k+1}} = h_{z-1,0}^{i,p}$. 
Again, by \cref{prop:prelimresultsA}, we have $h^i_{z,0}\nsim h^i_{z-1,0}$. 
We are just left to  show that $h^i_{-1,0}\nsim h^i_{0,0}$. Observe that 
$\lambda^{-2} h^i_{-1,0}\lambda^2 = h_{-1,0}^{i,p}$ and $\lambda^{-2} h^i_{0,0}\lambda^2 = h_{0,0}^{i,p^2}$. So $h^i_{-1,0}\nsim h^i_{0,0}$.
Therefore $h^i_{z,0}\nsim h^i_{z+1,0}$ for all $z\in \Z$.
Using this fact along with relation  $\zeta^{-1}h^i_{z,0}\zeta=h^i_{z+1,0}$ we can apply \cref{prop:prelimresultsB} to get that either 
\begin{gather*}
\cdots \ll  h^i_{z-1,0}\ll h^i_{z,0} \ll h^i_{z+1,0} \ll  \cdots\ll \zeta
\shortintertext{or}
\cdots \ll  h^i_{z+1,0}\ll h^i_{z,0} \ll h^i_{z-1,0} \ll  \cdots \ll\zeta.
\end{gather*}

Next, by the relation $\lambda^{-2^{-z}}h^i_{z,0}\lambda^{2^{-z}} = h_{z,0}^{i,p}$ and \cref{prop:prelimresultsA}, we can conclude that $h^i_{z,0}\ll \lambda^{2^{-z}}$. But since $\lambda\sim \lambda^{2^{-z}}$, in fact, $h^i_{z,0} \ll \lambda$ for all $z$.  We can also conclude that $\lambda\ll\zeta$ via $\zeta\lambda\zeta^{-1}=\lambda^2$. Putting everything together, it is the case that either
\begin{gather*}
\cdots \ll  h^i_{z-1,0}\ll h^i_{z,0} \ll h^i_{z+1,0} \ll  \cdots\ll \lambda\ll\zeta
\shortintertext{or}
\cdots \ll  h^i_{z+1,0}\ll h^i_{z,0} \ll h^i_{z-1,0} \ll  \cdots \ll \lambda \ll\zeta. \qedhere
\end{gather*}
\end{proof}

\begin{lemma} \label{lem:hihjnotequiv}
For all $i\neq j$, the elements $h_{u,0}^{i}$ and $h_{v,0}^{j}$ are not Archimedean equivalent for any $u,v\in\Z$.
\end{lemma}
\begin{proof}
  Suppose $u=v$. Note that $\lambda^{-2^{-u}} h^i_{u,0}\lambda^{2^{-u}}=h_{u,0}^{i,p_i}$ and $\lambda^{-2^{-v}} h^j_{v,0}\lambda^{2^{-v}}= h_{v,0}^{j,p_j}$. Since $p_i\neq p_j$, we must have $h^i_{u,0}\nsim h^j_{v,0}$  by \cref{prop:prelimresultsA}. 
  Now suppose $u\neq v$. We can fix an integer $k\geq 1$ such that either $k=u-v$ or $k=v-u$. Without loss of generality, say $k=u-v$. Then $\lambda^{-2^{-u+k}} h^i_{u,0}\lambda^{2^{-u+k}}=h_{u,0}^{i,p_i^{2^k}}$ and $\lambda^{-2^{-v}} h^j_{v,0}\lambda^{2^{-v}}=h_{v,0}^{j,p_j}$.
Since $-u+k=-v$ and $p_i^{2^k}\neq p_j$, it follows that $h^i_{u,0}\nsim h^j_{v,0}$ again from \cref{prop:prelimresultsA}.
\end{proof}
Using the above lemmas, we can deduce the following result.
\begin{lemma}\label{lem:allarchclasses}
The Archimedean classes of $G_n$ with respect to any order are 
\[ \{[1],[\lambda],[\zeta] \}\cup \{[h_{z,0}^{1}] \mid z\in\Z\}\cup \cdots \cup \{[h_{z,0}^{n}] \mid z\in\Z\}. \]
\end{lemma}
\begin{proof}
This follows since we can express each element of $G_n$ in a unique way. Suppose $g\in G_n$ is a nonidentity element. 
Then we can express $g$ uniquely in the form $h^1\cdots h^n \lambda^a\zeta^b$ where $h^i\in H^i$, $a\in A$ and $b\in \Z$. 
Furthermore, we can express each $h^i$ uniquely as $h^i_{z_1}\cdots h^i_{z_v}$  with $h^i_{z_j}\in H^i_{z_j}$ and $z_1<\cdots< z_v$.  
Now by \cref{prop:nprodclass}, the Archimedean class of $g$ will be whatever is the largest Archimedean class among the classes determined by the elements $h^i_{z_j}$, $\lambda^a$ and $\zeta^b$.
(Note that $\lambda^a\sim \lambda$ and $\zeta^b\sim \zeta$ whenever $a,b\neq 0$.)
\end{proof}

We state a definition before moving on. 
\begin{definition} \label{def:archmixed}
Let $A$ and $B$ be two subgroups of an ordered group. 
\begin{enumerate}
\item We will write $A\ll B$ to denote $a\ll b$ for all $a\in A\sm \{1\}$ and $b\in B\sm \{1\}$.
\item We will say that $A$ and $B$ are \textit{mixed} if $A\nll B$ and $B \nll A$.
\end{enumerate}
\end{definition}

\section{Invariants of orderings of \texorpdfstring{$G_n$}{Gn}} 
In this section we discuss a useful way to be able to describe all the orderings of $G_n$. 
Under any fixed order of $G_n$, by \cref{lem:Hiarchclasses}, we know that the Archimedean classes of each $H^i$ can only be ordered in exactly one of two ways. With this in mind we make the following definition.
\begin{definition}
Fix an ordering $<$ of $G_n$.
\begin{enumerate}
\item We will say that the direction of $H^i$ is \textit{positive} if it is the case that
\[ \cdots \ll h^i_{z-2,0}\ll  h^{i}_{z-1,0}\ll h^{i}_{z,0} \ll h^{i}_{z+1,0} \ll h^i_{z+2,0}\ll  \cdots. \]
\item We will say that the direction of $H^i$ is \textit{negative} if it is the case that 
\[  \cdots \ll h^i_{z+2,0}\ll h^{i}_{z+1,0}\ll h^{i}_{z,0} \ll h^{i}_{z-1,0} \ll h^i_{z-2,0}\ll \cdots.\]
\end{enumerate}
\end{definition}

We now introduce an auxiliary relation that will help us understand the orderings of $G_n$. Fix an ordering $<$ of $G_n$. Let $\bar n$ denote the set $\{1,\ldots,n\}$. 
Define $\usim_<$ on $\bar n$ by $i\sim_<j$ if and only if either $i=j$, or else $H^i$ and $H^j$ are mixed.
\begin{lemma} \label{lem:samedir}
If $i\sim_<j$, then $H^i$ and $H^j$ have the same direction.  That is, the direction of both is either positive or negative.
\end{lemma}
\begin{proof}
Suppose $i\neq j$. Then $H^i$ and $H^j$ are mixed and there exists integers $s,t,u,v$ such that 
$h_{u,0}^{i}\ll h_{v,0}^{j}$ and $h_{t,0}^{j}\ll h_{s,0}^{i}$.
Conjugate the relation 
$h_{t,0}^{j}\ll h_{s,0}^{i}$
by $\zeta^{v-t}$ to get $h^{j}_{v,0}\ll h^{i}_{s+v-t,0}$ and this implies $h^{i}_{u,0}\ll h^{j}_{v,0} \ll h^{i}_{s+v-t,0}$.

If the direction of $H^i$ is positive, then by \cref{lem:Hiarchclasses}, we can fix a largest possible integer $m\geq u$ such that $h^i_{m,0}\ll h^j_{v,0}$ but $h^i_{m+1,0} \nll h^j_{v,0}$.  Then $h^{i}_{m,0}\ll h^{j}_{v,0}\ll h^{i}_{m+1,0}$.
Similarly, if the direction of $H^i$ is negative, then by \cref{lem:Hiarchclasses}, we can fix a smallest possible integer $m\leq u$ such that $h^i_{m,0}\ll h^j_{v,0}$ but $h^i_{m-1,0} \nll h^j_{v,0}$.  Then $h^{i}_{m,0}\ll h^{j}_{v,0}\ll h^{i}_{m-1,0}$.
In any case, we can fix an integer $m$ such that either $h^{i}_{m,0}\ll h^{j}_{v,0}\ll h^{i}_{m+1,0}$ or $h^{i}_{m,0}\ll h^{j}_{v,0}\ll h^{i}_{m-1,0}$ depending on how the Archimedean classes of $H^i$ are ordered. 
If we conjugate these by $\zeta^{-m}$ we get that either
\begin{gather*}
h^{i}_{0,0} \ll h^{j}_{l,0} \ll h^{i}_{1,0}
\shortintertext{or}
h^{i}_{0,0} \ll h^{j}_{l,0} \ll h^{i}_{-1,0}
\end{gather*} where $l=v-m$. 
By conjugating the above two relations by various integral powers of $\zeta$ we will see that either
\begin{gather*}
\cdots \ll h^{i}_{0,0}\ll h^{j}_{l,0}\ll h^{i}_{1,0}\ll h^{j}_{l+1,0}\ll h^{i}_{2,0}\ll \cdots 
\shortintertext{or}
\cdots \ll h^{i}_{0,0}\ll h^{j}_{l,0}\ll h^{i}_{-1,0}\ll h^{j}_{l-1,0}\ll h^{i}_{-2,0}\ll \cdots 
\end{gather*}
holds. 
Thus we see that the Archimedean classes of $H^i$ and $H^j$ have the same direction.
\end{proof}

The previous proof also shows the following lemma stating how the Archimedean classes of two different $H^i$ and $H^j$ are ordered when $H^i$ and $H^j$ are mixed.
\begin{lemma} \label{lem:HiHjmixed}
Suppose $i\neq j$, and $H^i$ and $H^j$ are mixed. Then there exists an integer $u\in \Z$ such that either
\begin{gather*}
\cdots \ll h^{i}_{-1,0}\ll h^{j}_{u-1,0}\ll h^{i}_{0,0}\ll h^{j}_{u,0}\ll h^{i}_{1,0}\ll h^j_{u+1,0}\ll \cdots 
\shortintertext{or}
\cdots \ll h^{i}_{1,0}\ll h^{j}_{u+1,0}\ll h^{i}_{0,0}\ll h^{j}_{u,0}\ll h^{i}_{-1,0}\ll h^j_{u-1,0}\ll \cdots 
\end{gather*}
is true.
\end{lemma}
\begin{proof}
See proof of \cref{lem:samedir}.
\end{proof}

\begin{lemma} 
The relation ${\sim_<}$ is an equivalence relation.
\end{lemma}
\begin{proof}
It is easy to see that ${\sim_<}$ is both reflexive and symmetric. We just need to show that $\sim_<$ is transitive.  Suppose $i\sim_< j$ and $j\sim_< k$. Suppose $H^i,H^j$ and $H^k$ all have positive direction.
We will only prove this case, the other case when all have negative direction is handled similarly. 
By \cref{lem:HiHjmixed}, we know there exists integers $u$ and $v$ such that
\begin{gather*} 
h_{0,0}^{i} \ll h_{u,0}^{j} \ll h_{1,0}^{i} \ll h_{u+1,0}^{j}\ll h_{2,0}^{i} \\
\shortintertext{and}
h_{u,0}^{j} \ll h_{v,0}^{k} \ll h_{u+1,0}^{j}. 
\end{gather*}
It follows $h^{i}_{0,0} \ll h^{j}_{u,0} \ll h^{k}_{v,0}$  and $h_{v,0}^{k} \ll h_{u+1,0}^{j} \ll h_{2,0}^{i}$. Thus $H^i$ and $H^k$ are mixed, and $i\sim_< k$.
\end{proof}

We will refer to the direction of an equivalence class $[i]$ as being \textit{positive} or \textit{negative}, by which we will mean that the direction of $H^i$ is positive or negative, respectively. 
\begin{lemma} \label{lem:welldefined}
If $i\sim_< k$ and $j\sim_< l$, then $H^i\ll H^j\text{ if and only if }H^k \ll H^l$.
\end{lemma}
\begin{proof}
  Assume $H^i\ll H^j$. This implies that $i\nsim_< j$. Since $\sim_<$ is an equivalence relation, it follows that $k\nsim_< l$. Then either $H^k\ll H^l$ or $H^l\ll H^k$. 
By way of contradiction, suppose $H^l\ll H^k$. 
Since $j\sim_< l$ and $i\sim_< k$, there exists integers $u$ and $v$ such that $h_{0,0}^{j}\ll h_{u,0}^{l}$ and $h_{0,0}^{k}\ll h_{v,0}^{i}$. 
By assumption $H^l\ll H^k$ and so $h_{0,0}^{j}\ll h_{u,0}^{l}\ll h_{0,0}^{k}\ll h_{v,0}^{i}$. But this contradicts the initial assumption $H^i\ll H^j$.
Therefore it must be that $H^k\ll H^l$. A similar argument shows that $H^k\ll H^l$ implies $H^i\ll H^j$.
\end{proof}

Observe that \cref{lem:welldefined} shows that the $\ull$ relation on the Archimedean classes induces a total order $\lesssim$ on the set of equivalence classes of $\bar n$ under $\usim_<$. Given two equivalence classes $[i]$ and $[j]$, we can define $[i]\lesssim [j]$ if and only if $[i]=[j]$ or $H^i \ll H^j$.

\begin{lemma} \label{lem:mxclasses}
Let $C$ be an equivalence class under $\sim_<$ with $\abs C=k$ and $k\geq 2$.  Let $i_0$ be the least integer of $C$.
\begin{enumroman}
\item Suppose the direction of $H^i$ is positive for all $i\in C$.  Then there is an enumeration $i_0,i_1,\ldots,i_{k-1}$ of $C$ and integers $u_1,\ldots,u_{k-1}$ such that 
\[\cdots\ll h_{0,0}^{i_0} \ll h_{u_1,0}^{i_1} \ll \cdots \ll h_{u_{k-1},0}^{i_{k-1}} \ll h_{1,0}^{i_0}\ll h_{u_1+1,0}^{i_1} \ll \cdots \ll h_{u_{k-1}+1,0}^{i_{k-1}} \ll h_{2,0}^{i_0}\ll\cdots.\]
\item Suppose the direction of $H^i$ is negative for all $i\in C$.  Then there is an enumeration $i_0,i_1,\ldots,i_{k-1}$ of $C$ and integers $u_1,\ldots,u_{k-1}$ such that 
\[\cdots\ll h_{0,0}^{i_0} \ll h_{u_1,0}^{i_1} \ll \cdots \ll h_{u_{k-1},0}^{i_{k-1}} \ll h_{-1,0}^{i_0}\ll  h_{u_1-1,0}^{i_1} \ll \cdots \ll h_{u_{k-1}-1,0}^{i_{k-1}} \ll h_{-2,0}^{i_0}\ll\cdots.\]
\end{enumroman}
\end{lemma}
\begin{proof}
We prove only (i). By \cref{lem:HiHjmixed}, we know there exists integers $u_1,\ldots,u_{k-1}$ such that 
\[ h_{0,0}^{i_0} \ll h_{u_1,0}^{i_1} \ll h_{1,0}^{i_0}, \, \ldots \, ,
h_{0,0}^{i_0} \ll h_{u_{k-1},0}^{i_{k-1}} \ll h_{1,0}^{i_0}.\]
The elements $ h_{u_1,0}^{i_1}, \ldots, h_{u_{k-1},0}^{i_{k-1}}$ are linearly ordered with respect to $\ull$. So by reindexing as necessary, we can conclude that 
\[ h_{0,0}^{i_0} \ll h_{u_1,0}^{i_1} \ll \cdots \ll h_{u_{k-1},0}^{i_{k-1}} \ll h_{1,0}^{i_0}.\]
Conjugating the above relation by integral powers of $\zeta$ we get that
\[\cdots\ll h_{0,0}^{i_0} \ll h_{u_1,0}^{i_1} \ll \cdots \ll h_{u_{k-1},0}^{i_{k-1}} \ll h_{1,0}^{i_0}\ll h_{u_1+1,0}^{i_1} \ll \cdots \ll h_{u_{k-1}+1,0}^{i_{k-1}} \ll h_{2,0}^{i_0}\ll\cdots.\qedhere\]
\end{proof}

We next want to describe a collection of invariants that uniquely describe an ordering of $G_n$. Recall, we had fixed an ordering $<$ of $G_n$. We can define a \textit{positivity string} $\gamma= (\gamma_1,\ldots,\gamma_{n+2})$, with $\gamma_i\in \{-1,1\}$, that will encode the signs of the elements $h_{0,0}^{1},\ldots,h_{0,0}^{n},\lambda$ and $\zeta$ under $<$. 
The numbers in positions $1$ to $n$ will encode the signs of $h_{0,0}^{1},\ldots,h_{0,0}^{n}$, respectively. The number in position $n+1$ will encode the sign of $\lambda$ and the number in position $n+2$ will encode the sign of $\zeta$. A value of $-1$ will correspond to the element being negative and a value of $1$ will correspond to the element being positive.
It is worth pointing out that since $\zeta^{-z} h^i_{0,0} \zeta^z =h^i_{z,0}$, it suffices to specify the sign of $h^i_{0,0}$ and this determines the sign of $h^i_{z,0}$ for all $z\in \Z$. In particular for a fixed $i$, $h^i_{z,0}$ for all $z\in\Z$ have the same sign. This is an important fact that will help ensure that $G_n$ has only countably many orders. 

Fix the equivalence relation ${\sim_<}$ on $\bar n$ as defined above and suppose that we have $m$ many equivalence classes. We have the induced ordering relation $\lesssim$ on the equivalence classes of $\bar n$ under ${\sim_<}$.
When we say the $i$-th equivalence class, we will mean the $i$-th equivalence class with respect to the $\lesssim$ ordering. (So for example, the $m$-th equivalence class will refer to the greatest class under the $\lesssim$ order.)
Define a \textit{direction string} $\delta= (\delta_1,\ldots, \delta_m)$ where $\delta_i=-1$ if the direction of the $i$-th equivalence class is negative and $\delta_i=1$ if the direction of the $i$-th equivalence class is positive.
Furthermore, by \cref{lem:mxclasses}, for each equivalence class $[j]$ with at least two elements, we can assign to it an enumeration of its elements and a finite set of integers that encode a total ordering of the Archimedean classes of $\bigcup_{i\in [j]} H^i$.
For example, let $C=\{i_0,i_1,\ldots,i_{k-1}\}$ be an equivalence class with $k$ many elements and positive direction, we can fix a string of pairs of the form 
$\angles{i_1,u_1},\ldots,\angles{i_{k-1},u_{k-1}}$ such that the Archimedean classes of $H^{i_0},H^{i_1},\ldots H^{i_{k-1}}$ are totally ordered as encoded by this string in the following way:
\[\cdots\ll h_{0,0}^{i_0} \ll h_{u_1,0}^{i_1} \ll \cdots \ll h_{u_{k-1},0}^{i_{k-1}} \ll h_{1,0}^{i_0}\ll h_{u_1+1,0}^{i_1} \ll \cdots \ll h_{u_{k-1}+1,0}^{i_{k-1}} \ll h_{2,0}^{i_0}\ll\cdots.\]
To summarize, we have the following list of invariants we can assign to each order $<$ of $G_n$:
\begin{enumerate}[label={\textup{(\thesection.\Roman*)}}, ref={\textup{\thesection.\Roman*}}]
\item \label{invA}
A positivity string $\gamma= (\gamma_1,\ldots,\gamma_{n+2})$ that will encode the signs of the elements $h_{0,0}^{1},\ldots,h_{0,0}^{n},\lambda$ and $\zeta$ under $<$.  
\item \label{invB}
An equivalence relation ${\sim_<}$ on $\bar n$ with $m$ many equivalence classes.
\item \label{invC}
An order relation ${\lesssim}$ on the equivalence classes of $\bar n$ under ${\sim_<}$.  
\item \label{invD}
A direction string $\delta= (\delta_1,\ldots,\delta_m)$ where $\delta_i$ encodes if the $i$-th equivalence class is positive or negative. 
\item \label{invE}
For each equivalence class $[j]=\{i_0,i_1,\ldots,i_{k-1}\}$ with at least two elements, a string of pairs of the form $\angles{i_1,u_1},\ldots,\angles{i_{k-1},u_{k-1}}$ that encodes how the Archimedean classes of $\bigcup_{i\in [j]}H^i$ are totally ordered under the $\ull$ relation.
\end{enumerate}
Conversely, we can start by fixing these invariants and use them to build a unique ordering of $G_n$.
\begin{proposition}
A choice of invariants as in \eqref{invA}--\eqref{invE} uniquely determine an ordering of $G_n$.
\end{proposition}
\begin{proof}
We show how we can use the invariants to define a unique ordering $\prec$ of $G_n$.
Let $g\in G_n$ be an arbitrary nonidentity element. We can write $g=\rho\lambda^a\zeta^b$ for some unique $\rho\in P_n$, $\lambda^a \in A$ and $\zeta^b\in\Z$. We show how to determine whether $g$ is positive or negative under $\prec$. 
We make a simple observation about the orders of $G_n$.
By \cref{lem:Hiarchclasses}, since $\rho\ll \lambda^a \ll \zeta^b$, then we have that $g$  is positive if and only if  
\[ \zeta^b \text{ is positive; or }b=0\text{ and }\lambda^a\text{ is positive; or }a=b=0\text{ and }\rho\text{ is positive.} \]
Therefore all the orders of $G_n$ are lexicographical type orders.

We first look at $\gamma_{n+2}$ to determine the sign of $\zeta$. Suppose $\gamma_{n+2}=1$. If $b>0$, then $g$ is positive, and if $b<0$, then $g$ is negative. 
Next suppose $\gamma_{n+2}=-1$. If $b>0$, then $g$ is negative, and if $b<0$, then $g$ is positive. 
If $b=0$, then we look at $\gamma_{n+1}$. In an analogous fashion, if $\gamma_{n+1}=1$ and $a>0$, then $g$ is positive, and if $a<0$, then $g$ is negative.  If $\gamma_{n+1}=-1$ and $a>0$, then $g$ is negative and if $a<0$, then $g$ is positive. 
If $a=b=0$, then $g=\rho$. 
In this case, we can express $g$ as $h^{t_1}\cdots h^{t_s}$ where $h^{t_i}\in H^{t_i}$. 
We have now reduced to the case to showing how to order the subgroup $P_n$ given our invariants. 
By \cref{lem:allarchclasses}, the Archimedean classes of $P_n$ will always be the same with respect to any order that arises as a restriction of an order on $G_n$.
Informally, we want to use the information encoded by the invariants to linearly order the Archimedean classes of $P_n$ and this will allow us to order $P_n$.

Consider $P_n =\prod_{i=1}^n H^i=\prod_{i=1}^n \dsum_{z\in \Z}H_z^i$ and its index set $\bar n\times \Z=\{1,\ldots,n\}\times \Z$.
We want to linearly order the index set so as to encode the information given by the invariants \eqref{invB}--\eqref{invE} and then we will lexicographically order $P_n$ with respect to the defined ordering of the index set.
Suppose \eqref{invB} and \eqref{invC} give us that $\bar n=[l_1]\sqcup \cdots \sqcup [l_m]$ and $[l_1]\lesssim \cdots \lesssim [l_m]$.
Then we can start to order the index set as follows
\[ [l_1] \times \Z < \cdots < [l_m] \times \Z. \]
To be precise, we mean that if $[l_i]\lesssim [l_j]$ and $[l_i]\neq [l_j]$ then $(l,x)<(l',x')$ for all $l\in [l_i], l'\in [l_j]$ and $x,x'\in \Z$.
Let us further consider the equivalence class $[l_1]$ and say $[l_1]= \{k_0,k_1, \ldots, k_r\}$.
Assume that the direction of $[l_1]$ is given to be positive by \eqref{invD}, and $\angles{k_1,u_1}, \ldots,\angles{k_r,u_r}$ is the string of pairs given by \eqref{invE}.
Then we can further refine our ordering of the index set to be 
\[ \cdots< (k_0,0)<(k_1,u_1)<\cdots <(k_r,u_r)< (k_0,1) < (k_1,u_1+1)< \cdots <(k_r,u_r+1)<(k_0,2)<\cdots. \]
If the direction of $[l_1]$ is negative, then we will instead have 
\[ \cdots< (k_0,0)<(k_1,u_1)<\cdots <(k_r,u_r)< (k_0,-1) < (k_1,u_1-1)< \cdots <(k_r,u_r-1)<(k_0,-2)<\cdots. \]
We can follow a similar procedure for the rest of the equivalence classes of $\bar n$ and this will allow us to linearly order the index set $\bar n \times \Z$. 

We had reduced to the case when $g=h^{t_1}\cdots h^{t_s}$ where $h^{t_i}\in H^{t_i}$.
We can further express each $h^{t_i}$ in the form $h^{t_i}_{z_1}\cdots h^{t_i}_{z_v}$ where $h^{t_i}_{z_j}\in H^{t_i}_{z_j}$.
We want to define a pair $(t,z)\in \bar n\times \Z$ by 
\[(t,z)= \max\{(t_i,z_j)\in \bar n \times \Z \mid h^{t_i}_{z_j}\text{ appears as a substring in } g\}\] 
where we are taking the maximum with respect to the ordering $<$ of $\bar n\times \Z$.
Lastly, once we have defined such a pair $(t,z)$, to determine the sign of $g$ we consider the element $h^t_z$. We want to assign to $h^t_z$ a sign and that will be the sign of $g$. To set whether $h^t_z$ is positive or negative, we can use the positivity string $\gamma$ given by \eqref{invA} and the order described in \eqref{eq:Hizorder}. Note if $\gamma$ tells us that $h^t_{0,0}$ is negative, then we use the dual of the order in \eqref{eq:Hizorder}.

We finish off the proof by discussing why is the defined order on $P_n$ invariant under conjugation actions \eqref{eq:gpactionHi} and \eqref{eq:gpactionZHi}. 
For \eqref{eq:gpactionZHi}, conjugating $g=h^{t_1}\cdots h^{t_s}$ by some $\zeta^b \in \Z$ would mean that elements of the form $h^{t_i}_{z_j}$ that appear in $g$ will now be shifted to $h^{t_i}_{z_j+b}$ and the pair $(t,z)$ described above will now be defined as $(t,z+b)$. So to give a sign to $\zeta^{-b}g \zeta^b$ we will consider $h^{t}_{z+b}$. But notice that $h^t_z$ and $h^t_{z+b}$ will be given the same sign because of how the ordering in \eqref{eq:Hizorder} is defined. 
For \eqref{eq:gpactionHi}, if we conjugate $g=h^{t_1}\cdots h^{t_s}$ by some $\lambda^a\in A$, then $(t,z)$ will remain the same. So we need to look at $\lambda^{-a} h^t_z\lambda^a$ to give a sign to $\lambda^{-a} g\lambda^a$. If $h^t_z= h^{t,r_1}_{z,x_1}\cdots h^{t,r_v}_{z,x_v}$, then 
\[\lambda^{-a} h^t_z\lambda^a=  h^{t,r_1p_t^{y_1}}_{z,x_1+a2^z-y_1}\cdots h^{t,r_vp_t^{y_v}}_{z,x_v+a2^z-y_v}.\] 
Observe that 
\[ r_1p_t^{y_1}\cdot p_t^{x_1+a2^z-y_1}+\cdots+ r_vp_t^{y_v}\cdot p_t^{x_v+a2^z-y_v}= (r_1 p_t^{x_1}+\cdots+ r_v p_t^{x_v}) p_t^{a2^z}. \]
Since $p_t^{a2^z}$ is a positive real number, it will follow that $\lambda^{-a} h^t_z\lambda^a$  will be given the same sign as $h^t_z$ and, in turn, $\lambda^{-a} g\lambda^a$ will have the same sign as $g$.
\end{proof}

We mention that for the ordering $\prec$ described in the above proof, we can start with this ordering and then define the invariants \eqref{invA}--\eqref{invE} from it. In this case, we will find that we get exactly the same set of invariants. Therefore, fixing our set of invariants is the same as fixing an ordering of $G_n$ and so we see that this way of describing the orders of $G_n$ will actually completely describe any possible ordering of $G_n$. With this knowledge in hand, the proof of the next theorem follows readily.
\begin{theorem}
The group $G_n$ has exactly countably many distinct orders. In other words, the space of orders $\X(G_n)$ is countable.
\end{theorem}
\begin{proof}
We have already seen that the orders of $G_n$ are in correspondence with the collection of invariants \eqref{invA}--\eqref{invE}. So to count the number of orders of $G_n$, we simply need to count all the different collections of invariants we can fix. 
For \eqref{invA}, we are picking a finite string of length $n+2$. There are only finitely many such strings. For \eqref{invB} and \eqref{invC}, there are only finitely many different equivalence relations we can define on the set $\bar{n}$ and only finitely many different ways to order any collection of finitely many equivalence classes. For \eqref{invD}, we are picking a finite string of length at most $n$.
For \eqref{invE}, there are only finitely many different ways to enumerate any particular equivalence class, but we also have to pick some finite set of integers for each equivalence class. And here we see that we can make countably many different choices when picking our finite sets of integers.
Hence we have countably many different possibilities when fixing our set of invariants. Therefore $G_n$ has exactly countably many distinct orders.
\end{proof}

\section{Cantor-Bendixson rank of \texorpdfstring{$\X(G_n)$}{X(Gn)}} 
The goal of this section is to prove that the Cantor-Bendixson rank of $\X(G_n)$ is $n$. 
Recall that we say an equivalence relation $R_1$ is a proper refinement of $R_2$ if $R_1\neq R_2$ and $R_1\subseteq R_2$. This implies that each equivalence class of $R_1$ is contained in an equivalence class of $R_2$. 
\begin{definition}
Let $<_1$ and $<_2$ be two orders of $G_n$, and let $\sim_{<_1}$ and $\sim_{<_2}$ be the corresponding equivalence classes, respectively. We will say that $<_2$ is \textit{more mixed relative to} $<_1$ if $\sim_{<_1}$ is a proper refinement of $\sim_{<_2}$.
\end{definition}
We start with proving the following useful lemma.
\begin{lemma} \label{lem:ordermoremixed}
For every positive cone $P\in \X(G_n)$ there exists an open neighborhood $U_P$ of $P$ such that if $Q\in U_P$ and $Q\neq P$, then $Q$ corresponds to an order of $G_n$ that is more mixed relative to $P$.
\end{lemma}
\begin{proof}
Let $P\in \X(G_n)$ be a positive cone and let $<$ be the order corresponding to $P$. Suppose ${\sim_<} = [l_1] \sqcup \cdots \sqcup [l_m]$ and $[l_1]\lesssim \cdots \lesssim [l_m]$. We can also fix the positivity string $\gamma$ given by invariant \eqref{invA}. 
We will define a collection of open neighborhoods that will be denoted by $U_0,U_{l_1},\ldots,U_{l_m}$ and then lastly we will set $U_P= U_0\cap U_{l_1}\cap \cdots \cap U_{l_m}$.

First suppose $[l_i]=\{k\}$. If the direction of $[l_i]$ is positive, then define $U_{l_i}$ to be the open neighborhood determined by the string of inequalities $1< h_{0,0}^{k,\gamma_k}< h_{1,0}^{k,\gamma_k}$.
In other words, if $Q\in U_{l_i}$ is a positive cone and $\prec$ is the order corresponding to $Q$, then ${\prec}$ must satisfy that $1\prec h_{0,0}^{k,\gamma_k}\prec h_{1,0}^{k,\gamma_k}$.
Note that by \cref{prop:archlessthanpos}, any order that satisfies $1< h_{0,0}^{k,\gamma_k}< h_{1,0}^{k,\gamma_k}$ must also satisfy $h_{0,0}^k\ll h^k_{1,0}$.
If the direction of $[l_i]$ is negative, then define $U_{l_i}$ to be the open neighborhood determined by $1< h_{0,0}^{k,\gamma_k}< h_{-1,0}^{k,\gamma_k}$.

Now suppose it is the case that $[l_i]= \{k_0,k_1,\dots,k_r\}$. Assume that $\angles{k_1,u_1}, \ldots,\angles{k_r,u_r}$ is the string of pairs given by invariant \eqref{invE}. If the direction of $[l_i]$ is positive, then define $U_{l_i}$ to be determined by the following string of inequalities
\[ 1< h^{k_0,\gamma_{k_0}}_{0,0}< h^{k_1,\gamma_{k_1}}_{u_1,0}< \cdots< h^{k_r,\gamma_{k_r}}_{u_r,0}< h^{k_0,\gamma_{k_0}}_{1,0}< h^{k_1,\gamma_{k_1}}_{u_1+1,0}<  \cdots< h^{k_r,\gamma_{k_r}}_{u_r+1,0}< h^{k_0,\gamma_{k_0}}_{2,0}.\]
If the direction of $[l_i]$ is negative, then define $U_{l_i}$ by the string of inequalities 
\[ 1< h^{k_0,\gamma_{k_0}}_{0,0}< h^{k_1,\gamma_{k_1}}_{u_1,0}< \cdots< h^{k_r,\gamma_{k_r}}_{u_r,0}< h^{k_0,\gamma_{k_0}}_{-1,0}< h^{k_1,\gamma_{k_1}}_{u_1-1,0}<  \cdots< h^{k_r,\gamma_{k_r}}_{u_r-1,0}< h^{k_0,\gamma_{k_0}}_{-2,0}.\]
Next, let $U_0$ be the open neighborhood determined by the string of inequalities 
\[ 1< h^{l_1,\gamma_{l_1}}_{0,0}< \cdots< h^{l_m,\gamma_{l_m}}_{0,0}< \lambda^{\gamma_{n+1}}< \zeta^{\gamma_{n+2}}.\]
Let $U_P= U_0 \cap U_{l_1}\cap \dots\cap U_{l_m}$. We claim $U_P$ is our desired neighborhood.

Let $Q\in U_P\sm \{P\}$ and let ${\prec}$ be the order corresponding to $Q$. First, note that invariant $\eqref{invA}$ will be the same for $Q$ by our choice of $U_P$. This is because if $k\in [l_i]$, then $1\prec h^{k,\gamma_k}_{v,0}$ holds for some $v\in \Z$ where $v$ depends on the exact string of inequalities that determine membership in $U_{l_i}$. 
Recall that $h^i_{z,0}$ for all $z\in\Z$ have the same sign. If $\gamma_k=1$, then $h^k_{v,0}$ is positive and, in turn,  $h^k_{0,0}$ is positive. If $\gamma_k=-1$, then $h^k_{v,0}$  and $h^k_{0,0}$ are both negative. Similarly for $\lambda$ and $\zeta$, by our definition of $U_0$, we will force these two elements to have the same sign as they did under the positive cone $P$.

Next, by our choice of $U_P$ we must have that $\usim_<\sub \usim_\prec$. This is because when $[l_i]$ contains more than one element, any order in $U_{l_i}$ must satisfy that $H^{k_0},H^{k_1},\ldots, H^{k_r}$ are mixed. In addition, it must be that $\usim_<\psub \usim_\prec$. To see this, suppose $\usim_< = \usim_\prec$. 
Then observe that invariants \eqref{invC}, \eqref{invD} and \eqref{invE} must be the same between $<$ and $\prec$ since the string of inequalities that determine membership in each $U_{l_i}$ specify the direction of the equivalence class $[l_i]$, i.e.\ invariant \eqref{invD}, and also specify a finite set of integers that encode an ordering of the Archimedean classes of $\bigcup_{l\in [l_i]} H^l$, i.e.\ invariant \eqref{invE}. 
Furthermore, $U_0$ encodes invariant \eqref{invC}. Thus $Q=P$, a contradiction to our choice of $Q$. Hence, we can conclude that $\usim_<\psub \usim_\prec$ and so by definition $Q$ is more mixed relative to $P$.
\end{proof}
One more lemma to help streamline the proof of our main result.
\begin{lemma} \label{lem:allmixedisol}
Let $P\in \X(G_n)$ be a positive cone. If all of the $H^i$ are mixed in the order corresponding to $P$, then $P$ is an isolated point.
\end{lemma}
\begin{proof}
Let $P\in \X(G_n)$ be a positive cone and let $<$ be the order corresponding to $P$. Suppose that all of the $H^i$ are mixed under $<$, in other words, ${\sim_<}=\bar n= \{1,\ldots,n\}$. Let $U_P$ be the open neighborhood of $P$ given by \cref{lem:ordermoremixed}. We claim that $U_P=\{P \}$. 

Suppose $Q\in U_P$ and $Q\neq P$. Let $\prec$ be the order corresponding to $Q$. Then $\prec$ is more mixed relative to $<$ by \cref{lem:ordermoremixed}. This means that ${\sim_<}$ is a proper refinement of ${\sim_\prec}$ and so ${\sim_<}\psub {\sim_\prec}$. But we also have that for any equivalence relation $E$ on the set $\bar n$, $E\sub {\sim_<}$ since ${\sim_<}$ consists of a single equivalence class. So ${\sim_<}\psub {\sim_\prec}$ and ${\sim_\prec}\sub {\sim_<}$, a contradiction. Therefore there is no $Q\in U_P$ with $Q\neq P$. 
Thus $U_P=\{P\}$ and $P$ is an isolated point.
\end{proof}
We finally arrive at our main result for this section.
\begin{theorem}
The Cantor-Bendixson rank of $\X(G_n)$ is $n$.
\end{theorem}
\begin{proof}
We want to show that $\CB(\X(G_n))=n$. In fact, we will argue that $n$ is least positive integer such that $\X(G_n)^{(n)}=\emptyset$.  We first show that $\CB(\X(G_n))\leq n$. 
By \cref{lem:allmixedisol}, any order that has all of the $H^i$ mixed must be an isolated order.
Therefore $\X(G_n)'$ can only contain orders that have at most $n-1$ many of the subgroups $H^1,\ldots,H^n$ mixed and no orders where all of the $H^i$ are mixed.
Next, let $P\in \X(G_n)'$ where $n-1$ many of $H^1,\ldots,H^n$ are mixed. Let $U_P$ be the open neighborhood of $P$ given by \cref{lem:ordermoremixed}. Observe that all points in $U_P\sm \{P\}$ have all the $H^i$ mixed and so are isolated points in $\X(G_n)$. It follows that $U_P\cap \X(G_n)'=\{P\}$ and $P$ is an isolated point in $\X(G_n)'$.
Therefore in $\X(G_n)'$ any order that has $n-1$ many of $H^1,\ldots,H^n$ mixed will be isolated and so  $\X(G_n)^{(2)}$ can only contain orders that have at most $n-2$ many of $H^1,\ldots,H^n$ mixed.
Using \cref{lem:ordermoremixed} again, we can show that for any $Q\in \X(G_n)^{(2)}$ where $n-2$ many of $H^1,\ldots,H^n$ are mixed will be an isolated point in $\X(G_n)^{(2)}$ since all points in $U_Q\sm \{Q\}$ are isolated in $\X(G_n)'$.
Thus $\X(G_n)^{(3)}$ can only contain orders that have at most $n-3$ many of $H^1,\ldots,H^n$ mixed.
Continuing in this way, $\X(G_n)^{(n-1)}$ can only contain orders that have none of the $H^i$ mixed. By considering invariants \eqref{invA}--\eqref{invE}, it is not hard to see that $G_n$ has only finitely many orders with none of the $H^i$ mixed and so $\X(G_n)^{(n-1)}$ must be a finite set. Hence $\X(G_n)^{(n)}=\emptyset$ and $\CB(\X(G_n))\leq n$.

Next, we will show that there exists a point in $\X(G_n)$ with Cantor-Bendixson rank at least $n-1$ and this will imply that $\CB(\X(G_n))= n$.
This is so because if there exists some $P\in\X(G_n)$ with $\CB(P)\geq n-1$, then $\X(G_n)^{(n-1)}\neq \emptyset$. But since $\X(G_n)^{(n)}=\emptyset$, we will have that $\CB(\X(G_n))=n$, as desired. 

Let $\mathcal O_k\sub \X(G_n)$ for $2\leq k\leq n$ be defined as follows. An ordering ${<}\in \mathcal O_k$ if and only if $\zeta,\lambda$ and $h_{0,0}^{i}$ for all $i$ are positive under $<$; the direction of each $H^i$ is positive; the induced order $\ll$ on the Archimedean classes satisfies 
\[ H^{j}\ll H^{k+1} \ll\cdots \ll H^n \]
with $1\leq j\leq k$; and $<$ satisfies an inequality of the form
\[ h_{0,0}^{1} < h_{u_1,0}^{2}< \cdots < h_{u_{k-1},0}^{k} < h_{1,0}^{1} \]
for some $u_1,\ldots,u_{k-1}\in \Z$. 
Observe that all of the orders in $\mathcal O_k$ have $H^1,\ldots, H^{k}$ mixed.
Let $\mathcal O_1\sub \X(G_n)$ be defined as follows. An ordering ${<}\in \mathcal O_1$ if and only if 
$\zeta,\lambda$ and $h_{0,0}^{i}$ for all $i$ are positive under $<$; the direction of each $H^i$ is positive; and the induced order $\ll$ on the Archimedean classes satisfies 
\[ H^{1}\ll H^{2} \ll\cdots \ll H^n. \]
Now if ${\prec} \in \mathcal O_{n}$, then ${\prec}$ is an isolated order by \cref{lem:allmixedisol} and therefore $\CB({\prec})=0$.

Next, we claim that each order in $\mathcal O_k$ is a limit point of the set $\mathcal O_{k+1}$ for $k=1,\ldots,n-1$.
To show this, fix $k$ and let ${<}\in \mathcal O_k$. The order ${<}$ is described by the following set of invariants:
\begin{enumerate}[label={\textup{(\Roman*)}}]
\item The elements $\zeta,\lambda$ and $h^i_{0,0}$ for all $i$ are positive under $<$.
\item ${\sim_<} = \{1,2,\ldots,k\} \sqcup \{k+1\}\sqcup\ldots\sqcup \{n\}$.
\item $[k]\lesssim [k+1]\lesssim\cdots\lesssim [n]$.
\item The direction of each equivalence class is positive.
\item The pair of strings $\angles{2,u_1},\ldots, \angles{k,u_{k-1}}$.
\end{enumerate}
Fix a basic open set $U$ containing ${<}$. Fix a finite set of elements $g_1,\ldots,g_m$ that determine the basic open set $U$. Then $g_1,\ldots,g_m$ are all positive under $<$. We want to define an order ${\prec}\in \mathcal O_{k+1}\cap U$ such that ${\prec}\neq {<}$ and $g_1,\ldots,g_m$ are all positive under $\prec$.
We can write each $g_i$ as $h^{t_1}\cdots h^{t_s}\lambda^a \zeta^b$ with $h^{t_i}\in H^{t_i},\lambda^a\in A$ and $\zeta^b\in \Z$. Each $h^{t_i}$ can be written in the form $h^{t_i}_{z_1}\cdots h^{t_i}_{z_r}$ where $h^{t_i}_{z_j} \in H^{t_i}_{z_j}$.
Choose an integer $v>1$ large enough so that no element from $H^1_v\cup H^2_{v+u_1}\cup\cdots\cup H^k_{v+u_{k_1}}$ appears in any $g_i$ and choose an integer $w$ small enough so that no element from $H^{k+1}_w$ appears in any $g_i$.
Now define $\prec$ by the following invariants:
\begin{enumerate}[label={\textup{(\Roman*)}}]
\item The elements $\zeta,\lambda$ and $h^i_{0,0}$ for all $i$ are positive under $\prec$.
\item ${\sim_\prec} = \{1,2,\ldots,k,k+1\} \sqcup \{k+2\}\sqcup\ldots\sqcup \{n\}$.
\item $[k+1]\precsim [k+2]\precsim\cdots\precsim [n]$.
\item The direction of each equivalence class is positive.
\item The pair of strings $\angles{2,u_1},\ldots, \angles{k,u_{k-1}}, \angles{k+1,w-v}$.
\end{enumerate}
Our choice of $v$ and $w$ will ensure that the Archimedean classes under $\prec$ satisfy 
\[ h^1_{0,0}\ll h^2_{u_1,0}\ll \cdots\ll h^k_{u_{k-1}}\ll h^{k+1}_{w-v,0}\ll h^1_{1,0}\ll h^1_{v,0}\ll h^2_{v+u_1,0}\ll \cdots\ll h^k_{v+u_{k-1},0}\ll h^{k+1}_{w,0}\ll h^1_{v+1,0}. \]
In this new ordering $\prec$ the elements $g_1,\ldots,g_m$ will be positive because for each element of the form $h^{k+1}_y$ and $h^t_z$, for any $1\leq t\leq k$, that appears in some $g_i$, we will have that $h^t_z\ll h^{k+1}_y$. So the Archimedean classes that appear in $g_1,\ldots,g_m$ are ordered in the exact same way as they were under $<$. Thus $g_1,\ldots,g_m$ will remain positive under $\prec$ and we have proven our claim.

We now know that each $P\in \mathcal O_{n-1}$ is a limit point of the set $\mathcal O_n$. As a limit point of a set of rank $0$ points, we have that $\CB(P)\geq 1$ for all $P\in \mathcal O_{n-1}$. Similarly, we have that every point in $\mathcal O_{n-2}$ is a limit point of the set $\mathcal O_{n-1}$, and so $\CB(P)\geq 2$ for all $P \in \mathcal O_{n-2}$. 
Continuing along, we see that if $P\in \mathcal O_k$ for $1\leq k\leq n-1$, then $\CB(P)\geq n-k$. In particular, we have that every point in $\mathcal O_1$ has Cantor-Bendixson rank at least $n-1$. Thus there exists a point in $\X(G_n)$ with Cantor-Bendixson rank at least $n-1$ and this completes the proof.
\end{proof}

\bibliographystyle{abbrvnat}
\bibliography{refs.bib}

\end{document}